\documentclass[12pt,leqno]{article}
\usepackage[english,activeacute]{babel}
\usepackage{amsmath,amsfonts,amssymb,mathrsfs,amsthm}
\newcommand\R{\mathbb{R}}

\newcommand\Z{\mathbb{Z}}

\newcommand\bna{\begin{eqnarray*}}
\newcommand\ena{\end{eqnarray*}}

\newcommand\bnan{\begin{eqnarray}}
\newcommand\enan{\end{eqnarray}}

\newcommand\Xsbt{X^{s,b}_{T}}
\newcommand\Xsb{X^{s,b}}

\newcommand\Xzbt{X^{0,b}_{T}}

\newcommand\intt{\int_0^t}
\newcommand\intT{\int_0^T}

\newcommand\nor[2]{\left\|#1\right\|_{#2}}
\newcommand\norL[1]{\left\|#1\right\|_{L^2}}
\newcommand\sgn{\textnormal{sgn}}

\newcommand\tend[2]{\underset{#1 \to #2}{\longrightarrow}}

\newcommand\Tu{\mathbb{T}^1}

\newtheorem{thrm}{Theorem}[section]
\newtheorem{rmrk}{Remark}[section]
\newtheorem{lmm}{Lemma}[section]
\newtheorem{crllr}{Corollary}[section]
\newtheorem{prpstn}{Proposition}[section]

\author{Camille Laurent\thanks{Universit\'e Paris-Sud, B\^atiment 425, 91405 Orsay, France (camille.laurent@math.u-psud.fr).}}

\title{Global controllability and stabilization for the nonlinear Schr\"odinger equation on an interval}
\begin{document}
\maketitle
\begin{abstract}
We prove global internal controllability in large time for the nonlinear Schr\"odinger equation on a bounded interval with periodic, Dirichlet or Neumann conditions. Our strategy combines stabilization and local controllability near $0$. We use Bourgain spaces to prove this result on $L^2$. We also get a regularity result about the control if the data are assumed smoother.  
\end{abstract}
\vspace{0.2 cm}


\vspace{0.2 cm}{\bf AMS subject classifications.} 93B05, 93D15, 35Q55, 35A21 
%

%

%
%
\maketitle


\tableofcontents
\section*{Introduction}
In this article, we study the stabilization and exact controllability for the periodic one-dimensional nonlinear Schr\"odinger equation (NLS). 
\begin{eqnarray}
\label{eqncontrolnl}
\left\lbrace
\begin{array}{rcl}
i\partial_t u + \partial_x^2 u &=&\lambda |u|^2u \quad \textnormal{on}\quad [0,+\infty [\times \Tu\\
u(0)&=&u_{0} \in L^2(\Tu)
\end{array}
\right.
\end{eqnarray}
with $\lambda \in \R$.

The well posedness in such a low regularity was proved by J. Bourgain \cite{Bourgain}. The proof uses the so called Bourgain spaces $\Xsb$ to get local well posedness and the conservation of the $L^2$ mass for global existence.

The aim of this article is to prove exact internal controllability of system (\ref{eqncontrolnl}) in large time for a control supported in any small open subset of $\Tu$. We also extend these results to $]0,\pi[$ with Dirichlet or Neumann boundary conditions. The strategy follows the one of B. Dehman, P. G\'erard and G. Lebeau \cite{control-nl} where exact controllability in $H^1$ is proved for defocusing NLS on compact surfaces. Our result differs from this one because we obtain a control at a lower regularity. This allows to consider the focusing and defocusing equation and to use a different stabilization term, which seems more natural. Moreover, if the Cauchy data are smoother, that is $H^s$ with $s\geq 0$, the control we build on $L^2$ keeps that regularity, without any assumption on the size in $H^s$. Yet, in this low regularity, Strichartz inequality of \cite{Strichartz} does not provide uniform well posedness, and this forces us to use $\Xsb$ spaces.

The strategy is first to prove stabilization and to combine it with local exact controllability near $0$ to get null controllability. Then, we remark that the equation obtained by reversing time fulfills exactly the same properties and this allows to establish exact controllability. \\
Let $a=a(x)\in L^{\infty}(\Tu)$ real valued, the stabilization system we consider is 
\begin{eqnarray}
\label{eqndampedL2intro}
\left\lbrace
\begin{array}{rcl}
i\partial_t u + \partial_x^2 u +ia^2 u &=&\lambda|u|^2u \quad \textnormal{on}\quad[0,T]\times \Tu\\
u(0)&=&u_{0} \in L^2(\Tu).
\end{array}
\right.
\end{eqnarray}
The well posedness of this system will be proved in Section \ref{sectionexistence} and we can check that it satisfies the mass decay.
\bnan
\left\|u(t)\right\|_{L^2}^2-\left\|u(0)\right\|_{L^2}^2=-2\intt \left\|au(\tau)\right\|_{L^2}^2.
\enan
Our theorem states that we have an exponential decay.
\begin{thrm}
\label{thmstab}
Assume that $a(x)^2>\eta>0$ on some nonempty open set. Then, for every $R_0>0$, there exist $C>0$ and $\gamma>0$ such that inequality
$$\norL{u(t)} \leq Ce^{-\gamma t} \norL{u_0} \quad t>0$$
holds for every solution $u$ of system (\ref{eqndampedL2intro}) with initial data $u_0$ such that $\left\|u_0\right\|_{L^2}\leq R_0$.
\end{thrm}
Then, as a consequence of stabilization and local controllability near $0$ established in Section \ref{sectioncontrol0}, we obtain the following result.
\begin{thrm}
\label{thmcontrol}
For any nonempty open set $\omega \subset \Tu$ and $R_0>0$, there exist $T>0$ and $C>0$ such that for every $u_0$ and $u_1$ in $L^2(\Tu)$ with
$$ \norL{u_0} \leq R_0 \quad \textnormal{and}\quad \norL{u_1} \leq R_0$$
there exists a control $g\in C([0,T],L^2)$ with $\nor{g}{L^{\infty}([0,T],L^2)}\leq C $ supported in $[0,T]\times \omega$, such that the unique solution $u$ in $X^{0,b}_T$ to the Cauchy problem
\begin{eqnarray}
\label{eqncontrolintro}
\left\lbrace
\begin{array}{rcl}
i\partial_t u + \partial_x^2 u &=&\lambda|u|^2u +g\quad \textnormal{on}\quad[0,T]\times \Tu\\
u(0)&=&u_{0} \in L^2(\Tu)
\end{array}
\right.
\end{eqnarray}
satisfies $u(T)=u_1$.\\
Moreover, if $u_0$ and $u_1 \in H^s$, with $s\geq 0$, one can impose $g\in C([0,T],H^s)$.
\end{thrm}

We deduce the same results on $L^2(]0,\pi[)$ with the Dirichlet (respectively Neumann) Laplacian. To accomplish this, we use the identification of $D(-\Delta_D)$ (resp. $D(-\Delta_N)$) with the closed subspace of $H^2(\R/2\pi\Z)$ of odd (resp. even) functions. We only have to check along the proof that the control we build on $\Tu = \R/2\pi\Z$ remains odd (resp. even) if $u_0$ is so. The propagation of regularity for the control takes the form : if $u_0 \in D(-\Delta_D^s)$, then one can choose $g\in C([0,T], D(-\Delta_D^s))$ (and similarly for $\Delta_N$).

The continuity in time for $g$ is obtained with time cutoff at each stage : the stabilization term is brought to $0$ and the local control we build is identically zero at initial and final time. For example, if $u_0$ and $u_1$ are assumed in $C^{\infty}$, it allows to impose $u$ and $g$ in $C^{\infty}([0,T]\times \Tu)$.
 
The independence of $C$, $\gamma$ and the time of control $T$ on the bound $R_0$ are an open problem. Yet, it is an interesting fact that even if we want a control in $H^s$, the time of controllability only depends on the size of the data in $L^2$. However, it is unknown whether there is really a minimal time of controllability. This is in strong contrast with the linear case where exact controllability occurs in arbitrary small time and the conditions are only geometric for the open set $\omega$. For example, exact controllability is known to be true when Geometric Control Condition is realized, see G. Lebeau \cite{control-lin1}, but also for any open set $\omega$ of $\mathbb{T}^n$, see S. Jaffard \cite{Jaffard} and V. Komornik \cite{Komornik}. N. Burq and M. Zworski \cite{Burq} also proved the equivalence with a resolvent estimate. Moreover, some recent studies have analysed the explosion of the control cost when $T$ tends to $0$ : K.- D. Phung \cite{Phung} by reducing to the heat or wave equation, L. Miller \cite{Miller} with resolvent estimates, G. Tenenbaum and M. Tucsnak \cite{Tenenbaum} with number theoretic arguments.

Let us now describe briefly the main arguments of the proof of Theorem \ref{thmstab} and \ref{thmcontrol}. First, the functional spaces used are the Bourgain spaces which are especially suited for solving dispersive equations. In our problem, we use some multilinear estimates in $\Xsb$ (see the definition in Section \ref{sectionprop}). The first step is the following estimate for $b\geq 3/8$, uniformly for $T\leq 1$
\bnan
\label{inegXsbL4}
 \left\|u\right\|_{L^4([0,T]\times \Tu)}\leq C \left\|u\right\|_{X^{0,b}_T}.
 \enan
This was first proved by J. Bourgain in \cite{Bourgain}. A simpler proof, due to N. Tzvetkov, can be found in the book of T. Tao \cite{Tao} p 104. This allows to prove multilinear estimates in $\Xsb$, as follows.
\begin{lmm} 
\label{lemmemultilin}
For every $s\geq 0$, $b,b'\geq 3/8$, there exists $C_s$ independent on $T\leq 1$ such that for $u$ and $\tilde{u} \in X^{s,b}_T$, we have
\bnan
\label{multilinL2}
 \left\| |u|^2 u \right\|_{X^{s,-b'}_T} &\leq &C \left\|u\right\|^2_{X^{0,b}_T} \left\|u\right\|_{X^{s,b}_T}\\
\label{multilinL2Hs}
 \left\||u|^2u-|\tilde{u}|^2\tilde{u}\right\|_{X^{s,-b'}_T} &\leq &C \left(\left\|u\right\|^2_{\Xsbt}+ \left\|\tilde{u}\right\|^2_{\Xsbt} \right) \left\|u-\tilde{u}\right\|_{\Xsbt}.
\enan
\end{lmm}
This type of multilinear estimates was introduced in \cite{Bourgain}, but we refer to \cite{Bourgainlivre} p 107 where the estimates we need are stated during the proof of Theorem 2.1 chapter V. In the Appendix, we recall the proof and precise some dependence in $s$ of the estimates.

We prove the control near $0$ by a perturbative argument near the one of E. Zuazua \cite{zuazua}. We use the fixed point theorem of Picard to deduce our result from the linear control. The propagation of $H^s$ regularity from the state to the control is obtained using this property for the linear control and a local linear behavior. The idea comes from the work of B. Dehman and G. Lebeau \cite{DL} about the wave equation where only some smallness on a finite number of harmonics is required. A notable fact in our case is that no assumption of smallness is made on the $H^s$ norm. We only need the $L^2$ norm to be small. Yet, to obtain a bound independent on $s$, we have to make some estimates with constants independent on $s$. This will only be possible up to smoother terms, but this will be enough to conclude.

The proof of stabilization is more intricate. In a contradiction argument, following B. Dehman, G. Lebeau, E. Zuazua \cite{control_NLW} and \cite{control-nl}, we are led to prove the strong convergence to zero in $X^{0,b}_T$ of some weakly convergent sequence $(u_n)$ solution to damped NLS. In \cite{control-nl}, the authors use some linearisability property of NLS in $H^1$. Yet, this is false in the $L^2$ case. Moreover, as it was seen by L. Molinet in \cite{Molinet}, a weak limit $u$ of solutions of NLS is in general not necessarily solution of the same equation. Indeed, we have to proceed a little differently.

We first establish the strong convergence by some propagation of compactness. For a sequence $(u_n)$ weakly convergent to $0$ in $X^{0,b}_T$ satisfying
\begin{eqnarray*}
\left\lbrace
\begin{array}{c}
i\partial_t u_n + \partial_x^2 u_n \rightarrow 0 \quad \textnormal{in}\quad X^{-1+b,-b}_T\\
u_n\rightarrow 0 \quad \textnormal{in}\quad L^2([0,T]\times \omega),
\end{array}
\right.
\end{eqnarray*}
we prove that $u_n\rightarrow 0$ in $L^2_{loc}([0,T]\times \Tu)$. As the geometric control assumption is fulfilled, the propagation of compactness could be proved using microlocal defect measure introduced by P. G\'erard \cite{mesuredefaut}, adapting to $X^{s,b}$ spaces the argument of \cite{control-nl} inspired by C. Bardos and T. Masrour \cite{BardosMasrour}. In dimension $1$, the microlocal analysis is much simpler and we have chosen, for the convenience of the reader, to prove it with elementary arguments (even if the ideas are the same).

Once we know that the convergence is strong, we infer that the limit $u$ is solution to NLS. We use a classical unique continuation theorem to infer that it is $0$.
\begin{prpstn}
\label{uniquecontinuation}
For every $T>0$ and $\omega$ any nonempty open set of $\Tu$, the only solution in $C^{\infty}([0,T]\times \Tu)$ to the system 
\begin{eqnarray*}
\left\lbrace
\begin{array}{c}
i\partial_t u + \partial_x^2 u = b(t,x)u  \textnormal{ on } [0,T]\times \Tu\\
u=0 \textnormal{ on } [0,T]\times \omega
\end{array}
\right.
\end{eqnarray*}
where $b(t,x) \in C^{\infty}([0,T]\times \Tu)$ is the trivial one $u \equiv 0$.
\end{prpstn}

This was proved by Isakov \cite{Isakov} (see Corollary 6.1) using Carleman estimates.

Yet, the weak limit a priori belongs to $X^{0,b}_T$. Therefore, to apply Proposition \ref{uniquecontinuation}, we need $u$ smooth enough. We prove that a solution of NLS with $u\in C^{\infty}([0,T]\times \omega)$ is actually smooth. The proof is an adaptation to the $\Xsb$ spaces of propagation results of microlocal regularity coming from \cite{control-nl}. Again, we present it in such a way that no knowledge of microlocal analysis is necessary, even if the ideas deeply come from this theory. 

\bigskip

While writing this article, we learnt that L. Rosier and B. Y. Zhang \cite{RosierZhang} independently obtained a result of local controllability of NLS near $0$.

\bigskip

\textbf{Notation} Denote $D^r$ the operator defined on $\mathcal{D'}(\Tu)$ by
\bnan
\begin{array}{rclc}
\label{Dr}
\widehat{D^r u}(n)&=& \sgn(n)|n|^r \widehat{u}(n) & \quad \textnormal{if}\quad n\neq 0\\
&=& \widehat{u}(0) & \quad \textnormal{if} \quad n= 0.
\end{array}
\enan

\bigskip

In this article, $b$ and $b'$ will be two constants, fixed for the rest of the article, such that $1>b+b'$, $b>1/2>b'$, and estimates (\ref{multilinL2}) and (\ref{multilinL2Hs}) hold, see Lemma \ref{gainint} below for the justification of these assumptions.

$C$ will denote any absolute constant whose value could change along the article. It could actually depend on $s$. Yet, when the dependence on $s$ will be needed, this will be announced and we will denote $C$ if it is independent on $s$ and $C_s$ otherwise.

\bigskip

{\it Acknowledgements.} The author deeply thanks his adviser Patrick G\'erard for attracting his attention to this problem and for helpful discussions and encouragements.
\section{Some properties of $\Xsb$ spaces}
\label{sectionprop}
We equip the Sobolev space $H^s(\Tu)$ with the norm 
$$\left\|u\right\|^2_{H^s}=\left\|D^s u\right\|^2_{L^2}=\left|\widehat{u}(0)\right|^2+\sum_{k\neq 0}\left| k\right|^{2s} \left|\widehat{u}(k)\right|^2.$$
The Bourgain space $\Xsb$ is equipped with the norm
 \bna
 \left\|u\right\|^2_{X^{s,b}}&=& \left\|\widehat{u}(.,0)\right\|^2_{H^b(\R)}+\sum_{k}\int_{\R}\left| k\right|^{2s}\left\langle \tau+ k^2 \right\rangle^{2b} \left|\widehat{\widehat{u}}(\tau,k)\right|^2 d\tau\\
 &=&\left\|u^{\#}\right\|^2_{H^{b}(\R,H^s(\Tu))}
 \ena
 where $\left\langle .\right\rangle =\sqrt{1+|.|^2}$, $u=u(t,x)$, $t\in \R$, $x\in \Tu$, and $u^{\#}(t)=e^{-it\partial_x^2}u(t)$. $ \widehat{\widehat{u}}(\tau,k)$ denotes the Fourier transform of $u$ with respect to the time variable (indice $\tau$) and space variable (indice $k$). $\widehat{u}(t,k)$ denotes the Fourier transform in space variable.\\
$\Xsbt$ is the associated restriction space, with the norm
\bna 
\left\|u\right\|_{\Xsbt}=\inf \left\{ \left\| \tilde{u}\right\|_{\Xsb} \left| \tilde{u}=u \textnormal{   on   } [0,T]\times \Tu  \right. \right\}.
\ena
Let us study the stability of the $\Xsb$ spaces with respect to some particular operations.
\begin{lmm}
\label{lemmetps}
Let $\psi \in C^{\infty}_0(\R)$ and $u\in \Xsb$ then $\psi(t) u \in \Xsb$.\\
If $u\in \Xsbt$ then we have $\psi(t) u \in \Xsbt$.
\end{lmm}
\begin{proof}
We write
$$\left\|\psi u\right\|_{X^{s,b}}= \left\|e^{-it\partial_x^2}\psi(t)u\right\|_{H^{b}(H^s)}= \left\|\psi u^{\#}\right\|_{H^{b}(H^s)}\leq C  \left\|u^{\#}\right\|_{H^{b}(H^s)}\leq C\left\|u\right\|_{X^{s,b}}.$$
We get the second result by applying the first one on any extension of $u$ and taking the infinimum.\end{proof}

We easily get that $D^r$ (using notation (\ref{Dr})) maps any $X^{s,b}$ into $X^{s-r,b}$. In the case of multiplication by $C^{\infty}(\Tu)$ function, we have to deal with a loss in $\Xsb$ regularity compared to what we could expect. Some regularity in the index $b$ is lost, due to the fact that multiplication does not keep the structure in time of the harmonics. This loss is unavoidable : take  $u_n=\psi(t)e^{inx}e^{in^2t}$ (where $\psi \in C^{\infty}_0(\R)$ equal to $1$ on $[-1,1]$) which is uniformly bounded in $X^{0,b}$ for every $b\geq 0$. Yet, if we consider the operator of multiplication by $e^{ix}$, we get $\left\|e^{ix}u_n\right\|_{X^{0,b}} \approx n^b$. We can prove that our example is the worst one.
\begin{lmm}
\label{lemmepseudoxsb}
Let $ -1 \leq b \leq 1$, $s\in \R$ and $\varphi \in C^{\infty}(\Tu)$. Then, if $u\in \Xsb$ we have $\varphi(x) u \in X^{s-|b|,b}$.\\
Similarly, multiplication by $\varphi$ maps $\Xsbt$ into $X^{s-|b|,b}_T$.
\end{lmm}
\begin{proof}
We first deal with the two cases $b=0$ and $b=1$ and we will conclude by interpolation and duality.\\
For $b=0$, $X^{s,0}=L^2(\R,H^s)$ and the result is obvious.\\
For $b=1$, we have $u\in X^{s,1}$ if and only if
 $$u \in L^2(\R,H^s) \textnormal{ and } i\partial_t u+\partial_x^2 u \in L^2(\R,H^s)$$
 with the norm
 $$ \left\| u\right\|^2_{X^{s,1}}= \left\|u\right\|^2_{L^2(\R,H^s)}+\left\|i\partial_t u+\partial_x^2 u\right\|^2_{L^2(\R,H^s)}.$$
Then, we have 
\bna
\left\|\varphi(x) u\right\|^2_{X^{s-1,1}}&=&\left\|\varphi u\right\|^2_{L^2(\R,H^{s-1})}+ \left\|i\partial_t (\varphi u)+\partial_x^2(\varphi u)\right\|^2_{L^2(\R,H^{s-1})}\\
&\leq& C\left(\left\|u\right\|^2_{L^2(\R,H^{s-1})}+ \left\|\varphi \left(i\partial_t u+\partial_x^2 u\right)\right\|^2_{L^2(\R,H^{s-1})}\right.\\
& &\left.+ \left\|\left[\varphi,\partial_x^2 \right]u\right\|^2_{L^2(\R,H^{s-1})}\right)\\
&\leq & C\left(\left\|u\right\|^2_{L^2(\R,H^{s-1})}+ \left\|i\partial_t u+\partial_x^2 u\right\|^2_{L^2(\R,H^{s-1})}+ \left\|u\right\|^2_{L^2(\R,H^{s})}\right)\\
&\leq & C\left\|u\right\|^2_{X^{s,1}}.
\ena
Here, we have used that $\left[\varphi,\partial_x^2 \right]=-2(\partial_x \varphi) \partial_x- (\partial^2_x \varphi)$ is a differential operator of order $1$. To conclude, we prove that $\Xsb$ spaces are in interpolation. For that, we consider $\Xsb$ as a weighted $L^2(\R \times \Z,\mu \otimes \delta )$ spaces, where $\mu$ is the Lebesgues measure on $\R$ and $\delta$ is the discret measure on $\Z$. Using the Fourier transform, we can interpret $\Xsb$ as the weighted $L^2$ space
$$L^2\left(\R \times \Z,w_{s,b}(\tau,k)\mu \otimes \delta \right) $$
where $w_{s,b}(\tau,k)=\left| k\right|_{\wr}^{2s} \left\langle \tau+k^2\right\rangle^{2b}$. Here, we denote 
\bnan
\label{notationnorm}
\left| k\right|_{\wr}=\left|k\right| \textnormal{ if } k\neq 0 \textnormal{ and } 1 \textnormal{ otherwise.}
\enan
Then, we use the complex interpolation theorem of Stein-Weiss for weighted $L^p$ spaces (see \cite{Bergh} p 114 ) : for $0<\theta <1$
$$\left(X^{s,0},X^{s',1}\right)_{[\theta]} \approx L^2\left(\R \times \Z,\left| k\right|_{\wr}^{2s(1-\theta)+2s'\theta} \left\langle \tau+k^2\right\rangle^{2\theta	}\mu \otimes \delta \right)\approx X^{s(1-\theta)+s'\theta,\theta}.$$
Since $\varphi$ maps $X^{s,0}$ into $X^{s,0}$ and $X^{s,1}$ into $X^{s-1,1}$, we conclude that for $0\leq b \leq 1$, $\varphi$ maps $\Xsb=\left(X^{s,0},X^{s,1}\right)_{[b]}$ into  \mbox{$\left(X^{s,0},X^{s-1,1}\right)_{[b]} =X^{s-b,b}$}  which yields the $b$ loss of regularity as announced.\\
Then, by duality, this also implies that for $0\leq b \leq 1$, $\varphi(x)$ maps $X^{-s+b,-b}$ into $X^{-s,-b}$. As there is no assumption on $s\in \R$, we also have the result for $-1\leq b \leq 0$ with a loss $-b=|b|$.\\
To get the same result for the restriction spaces $\Xsbt$, we write the estimate for an extension $\tilde{u}$ of $u$, which yields
\bna
\left\|\varphi u\right\|_{X^{s-|b|,b}_T} \leq \left\|\varphi \tilde{u}\right\|_{X^{s-|b|,b}} \leq C  \left\|\tilde{u}\right\|_{X^{s,b}}. 
\ena
Taking the infinimum on all the $\tilde{u}$, we get the claimed result.\end{proof}

\bigskip

We will also use (see \cite{ginibre} or \cite{Bourgain})
\begin{lmm}
\label{gainint}
Let $(b,b')$ satisfying
\begin{eqnarray}
0<b'<\frac{1}{2}<b,~~~~b+b'\leq 1. 
\end{eqnarray}
If we note $F(t)=\Psi\left(\frac{t}{T}\right)\intt f(t')dt'$, we have for $T\leq 1$
\bna
\left\|F\right\|_{H^b} \leq CT^{1-b-b'}\left\|f\right\|_{H^{-b'}}.
\ena
\end{lmm}
 
In the futur aim of using a boot-strap argument, we will need some continuity in $T$ of the $X^{s,b}_T$ norm of a fixed function : 
\begin{lmm}
\label{continuiteXsbt}
Let $0<b<1$ and $u$ in $X^{s,b}$ then the function
\begin{eqnarray*}
\left\lbrace
\begin{array}{rcrcl}
f&:&]0,T]& \longrightarrow &  \R \\
 & &t    & \longmapsto     & \left\|u\right\|_{X^{s,b}_{t}}
\end{array}
\right.
\end{eqnarray*} 
 is continuous.
Moreover, if $b>1/2$, there exists $C_b$ such that 
$$\lim_{t\rightarrow 0} f(t) \leq C_b \nor{u(0)}{H^s}.$$
 \end{lmm}
\begin{proof}
By reasoning on each component on the basis, we are led to prove the result in $H^b(\R)$. The most difficult case is the limit near $0$. It suffices to prove that if $u\in H^b(\R)$, with $b>1/2$, satisfies $u(0)=0$, and $\Psi \in C^{\infty}_0(\R)$ with $\Psi(0)=1$, then
$$\Psi\left(\frac{t}{T}\right)u \tend{T}{0} 0 \quad \textnormal{in} \quad H^b.$$
Indeed, such a function $u$ can be written $\intt f$ with $f\in H^{b-1}$. Then, Lemma \ref{gainint} gives the result we want if $u\in H^{b+\varepsilon}$. Nevertheless, if we only have $u\in H^b$, $\Psi(\frac{t}{T})u$ is uniformly bounded. We conclude by a density argument.\end{proof}

\bigskip

The following lemma will be useful to control solutions on large intervals that will be obtained by piecing together solutions on smaller ones. We state it without proof.
\begin{lmm}
\label{lemmerecouvrement}
Let $0<b<1$. If $\bigcup ]a_k,b_{k}[$ is a finite covering of $[0,1]$, then there exists a constant $C$ depending only of the covering such that for every $u\in X^{s,b}$
\begin{eqnarray*}
\left\|u\right\|_{X^{s,b}_{[0,1]}}\leq C\sum_k \left\|u\right\|_{X^{s,b}_{[a_k,b_{k}]}}.
\end{eqnarray*}
\end{lmm}
Finally, we have the following Rellich type lemma
\begin{lmm}
\label{injectioncompacte}
For every $\delta >0$, $\eta>0$, $s$, $b\in \R$ and $T>0$, we have
$$X^{s+\eta,b+\delta}_T \subset  X^{s,b}_T$$
with compact imbedding.
\end{lmm}
\section{Existence of a solution to NLS with source and damping term}
\label{sectionexistence}
\begin{thrm}
\label{thmexistenceNl}
Let $T>0$, $s\geq 0$, $\lambda \in \R$ and $a\in C^{\infty}(\Tu)$, $\varphi \in  C_0^{\infty}(\R)$ taking real values.\\
For every $g\in L^2([-T,T],H^s)$ and $u_0 \in H^s$, there exists a unique solution $u$ in $X^{s,b}_T$ to 
\begin{eqnarray}
\label{dampedeqn}
\left\lbrace
\begin{array}{rcl}
i\partial_t u + \partial_x^2 u +i\varphi(t)^2a(x)^2 u&=& \lambda |u|^2u+g \textnormal{ on } [-T,T]\times \Tu\\
u(0)&=&u_{0} \in H^s
\end{array}
\right.
\end{eqnarray}
Moreover the flow map $$
\begin{array}{rcrcl}
F &:& H^s(\Tu) \times L^2([-T,T],H^s(\Tu))&\rightarrow & X^{s,b}_{[-T,T]}\\
           & & (u_0,g) &\mapsto   &  u
\end{array}
$$ is Lipschitz on every bounded subset.\\
The same results occur for $s=0$ with the weaker assumption $a\in L^{\infty}(\Tu)$.
\end{thrm}

\begin{proof} It is strongly inspired by Bourgain's one (see \cite{Bourgain}, \cite{Bourgainlivre} and \cite{ginibre}).
First, we notice that if $g\in L^2([-T,T],H^s)$, it also belongs to $X^{s,-b'}_T$ as $b'\geq 0$. We restrict ourself to positive times. The solution on $[-T,0]$ is obtained similarly. The distinction on the case $s=0$ and $s>0$ for the regularity assumption on $a$ will appear along the proof with the following statement : with the assumptions of the Theorem, multiplication by $a$ maps $X^{s,0}=L^2([0,T],H^s)$ into itself.\\
We consider the functional
 $$\Phi(u)(t)=e^{it\partial_x^2}u_0-i\int_0^t e^{i(t-\tau)\partial_x^2}\left[-ia^2\varphi^2 u+\lambda \left|u\right|^2u+ g\right](\tau) d\tau .$$
We will apply a fixed point argument on the Banach space $\Xsbt$.
Let $\psi \in C^{\infty}_0(\R)$ be equal to $1$ on $[-1,1]$. Then by construction, (see \cite{ginibre}) :
$$\left\|\psi(t)e^{it\partial_x^2}u_0\right\|_{\Xsb}=\left\|\psi\right\|_{H^{b}(\R)} \left\|u_0\right\|_{H^s}.$$
Indeed, for $T\leq 1$ we have
$$\left\|e^{it\partial_x^2}u_0\right\|_{\Xsbt}\leq C \left\|u_0\right\|_{H^s}.$$
The one dimensional estimate of Lemma \ref{gainint} implies
$$ \left\|\psi(t/T)\intt e^{i(t-\tau)\partial_x^2}F(\tau)\right\|_{\Xsb}\leq CT^{1-b-b'}\left\|F\right\|_{X^{s,-b'}}$$
and then
\bnan
\label{inegprincip}
&&\left\|\int_0^t e^{i(t-\tau)\partial_x^2}\left[-ia^2\varphi^2 u+\lambda \left|u\right|^2u+ g\right](\tau) d\tau\right\|_{\Xsbt} \nonumber\\
& \leq&  CT^{1-b-b'}\left\|-ia^2\varphi^2 u+\lambda \left|u\right|^2u+ g\right\|_{X^{s,-b'}_T} \nonumber\\
&\leq & CT^{1-b-b'}\left[\left\|\varphi^2a^2 u\right\|_{X^{s,0}_T}+\left\|\left|u\right|^2u\right\|_{X^{s,-b'}_T}+\left\|g\right\|_{X^{s,-b'}_T}\right] \nonumber \\
& \leq & CT^{1-b-b'}\left\|u\right\|_{X^{s,b}_T}\left(1+\left\|u\right\|_{X^{0,b}_T}^2\right)+\left\|g\right\|_{X^{s,-b'}_T}.
\enan
Thus
\bnan
\label{tameestimateDuhamel}
\left\|\Phi(u)\right\|_{\Xsbt}\leq C \left\|u_0\right\|_{H^s}+C\left\|g\right\|_{X^{s,-b'}_T}+CT^{1-b-b'}\left\|u\right\|_{X^{s,b}_T}\left(1+\left\|u\right\|_{X^{0,b}_T}^2\right)
\enan
and similarly,
\bnan
\label{majdiff}
\left\|\Phi(u)-\Phi(\tilde{u})\right\|_{\Xsbt}\leq CT^{1-b-b'}\left\|u-\tilde{u}\right\|_{\Xsbt}\left(1+\left\|u\right\|_{\Xsbt}^2+\left\|\tilde{u}\right\|_{\Xsbt}^2\right).
\enan
These estimates imply that if $T$ is chosen small enough $\Phi$ is a contraction on a suitable ball of $\Xsbt$. \\
Moreover, we have uniqueness in the class $\Xsbt$ for the Duhamel equation. To get the uniqueness in $\Xsbt$ for the Schr\"odinger equation itself, we prove that every solution $u$ in $\Xsbt$ of equation (\ref{dampedeqn}) in the distributional sense is also solution of the integral equation. Let us put 
$$w(t)=e^{it\partial_x^2}u_0-i\int_0^t e^{i(t-\tau)\partial_x^2}\left[-i\varphi^2a^2 u+\lambda \left|u\right|^2u+ g\right](\tau) d\tau .$$
As $u\in \Xsbt$, we have $\left|u\right|^2u \in X^{s,-b'}_T$ and since $b'<1/2$, we infer
\bna
\partial_t \left[\int_0^t e^{-i\tau \partial_x^2}\left[-ia^2\varphi^2 u+\lambda \left|u\right|^2u+ g\right](\tau) d\tau  \right]\\ = e^{-it \partial_x^2}\left[-i\varphi^2a^2\varphi^2 u+\lambda \left|u\right|^2u+ g\right](t)
\ena
in the distributional sense which implies that $w$ is solution of 
$$ i\partial_t w + \partial_x^2 w +i\varphi^2a^2\varphi^2 u=\lambda \left|u\right|^2u+ g .$$
Then, $r=e^{-it\partial_x^2}(u-w)$ is solution of $\partial _t r =0$ and $r(0)=0$. Hence, $r=0$ and  $u$ is solution of the integral equation. Actually, the above proof also gives that the solution $u$ of the integral equation is also solution in the distributional sense.\\
We also prove propagation of regularity.\\
If $u_0\in H^s$, with $s>0$, we have an existence time $T$ for the solution in $\Xzbt$ and another time $\tilde{T}$ for the existence in $X^{s,b}_{\widetilde{T}}$. By uniqueness in $\Xzbt$, the two solutions are the same on $[0,\tilde{T}]$. If we assume $\tilde{T}<T$, we have the explosion of $\left\|u(t,.)\right\|_{H^s}$ as $t$ tends to $\tilde{T}$ whereas $\left\|u(t,.)\right\|_{L^2}$ remains bounded on this interval. Using local existence in $L^2$ and Lemma \ref{lemmerecouvrement}, we easily get that $\left\|u\right\|_{X^{0,b}_{\widetilde{T}}}$ is finite. Then, using tame estimate (\ref{tameestimateDuhamel}) on a subinterval $[\widetilde{T}-\varepsilon,\widetilde{T}]$, with $\varepsilon$ small enough such that $C\varepsilon^{1-b-b'}\left(1+\left\|u\right\|_{X^{0,b}_{[\widetilde{T}-\varepsilon,\widetilde{T}]}}^2\right)<1/2$, we obtain 
\bna
\left\|u\right\|_{X^{s,b}_{[\widetilde{T}-\varepsilon,\widetilde{T}]}}\leq C \left\|u(T-\varepsilon)\right\|_{H^s}+\left\|g\right\|_{X^{s,-b'}_{[\widetilde{T}-\varepsilon,\widetilde{T}]}}
\ena
We conclude that $u\in X^{s,b}_{\widetilde{T}}$, which contradicts the explosion of $\left\|u(t,.)\right\|_{H^s}$ near $\tilde{T}$. Therefore, the time of existence is the same for every $s\geq 0$.\\
Next, we use $L^2$ energy estimates to get global existence in $\Xzbt$ and indeed in $\Xsbt$. By multiplying equation (\ref{dampedeqn}) by $\overline{u}$, taking imaginary part and integrating, we get 
\bna
\left\|u(t)\right\|_{L^2}^2-\left\|u(0)\right\|_{L^2}^2=-2\intt \left\|a\varphi(\tau) u(\tau)\right\|_{L^2}^2+2\Im \intt \int_{\Tu}g \bar{u} 
\ena
\bna
\left\|u(t)\right\|_{L^2}^2& \leq & \left\|u(0)\right\|_{L^2}^2+C \intt \left\|u(\tau)\right\|_{L^2}^2+\intt \left\|u(\tau)\right\|_{L^2}\left\|g(\tau)\right\|_{L^2}\\
& \leq & \left\|u(0)\right\|_{L^2}^2+C \intt \left\|u(\tau)\right\|_{L^2}^2d\tau +C\left\|g\right\|_{L^2([-T,T],L^2)}^2.
\ena
Then, by Gronwall inequality, we have
\bnan
\label{inegGronwallL2}
\left\|u(t)\right\|_{L^2}^2\leq  C\left(\left\|u(0)\right\|_{L^2}^2+\left\|g\right\|^2_{L^2([-T,T],L^2)}\right)e^{C|t|}.
\enan
This ensures that the $L^2$ norm remains bounded and the solution $u$ is global in time.\\
For the continuity of the flow, we use a slight modification of estimate (\ref{majdiff}) for two solutions $u$ and $\tilde{u}$
\bna
\left\|u-\tilde{u}\right\|_{X^{s,b}_{T}}\leq C\left\|u(0)-\tilde{u}(0)\right\|_{H^s}+C\left\|g-\tilde{g}\right\|_{X^{s,-b'}_{T}}\\
+ CT^{1-b-b'}\left\|u-\tilde{u}\right\|_{X^{s,b}_{T}}\left(1+\left\|u\right\|_{X^{s,b}_{T}}^2+\left\|\tilde{u}\right\|_{X^{s,b}_{T}}^2\right).
\ena
Then, for $T$ small enough (depending on the size of $u_0$, $\widetilde{u_0}$, $g$ and $\tilde{g}$), we get
\bna
\left\|u-\tilde{u}\right\|_{X^{s,b}_{T}}\leq C\left\|u(0)-\tilde{u}(0)\right\|_{H^s}+C\left\|g-\tilde{g}\right\|_{X^{s,-b'}_{T}}.
\ena 
Then, we just have to piece solutions together on small intervals. Using the control of the $X^{s,b}_{T}$ norm on $L^{\infty}([0,T],H^s)$ and Lemma \ref{lemmerecouvrement}, we get that $F$ is Lipschitz on bounded sets for arbitrary $T$.
\end{proof}
\bigskip

After this point and until the end of the proof of local controllability, we will express the dependence on $s$ of the constants by writing them $C_s$ or $C(.)$ if some other dependence is considered. $b$, $b'$, $\lambda$, $a$ and $\varphi$ being fixed, we will not write the dependence of constants in these variables.\\ 
The following Propositions establish a linear behavior on bounded sets of $L^2$.
\begin{prpstn}
\label{linearbehavior}
For every $T>0$, $\eta>0$ and $s\geq 0$, there exists $C(T,\eta,s)$ such that for every $u\in \Xsbt$ solution of (\ref{dampedeqn}) with $\left\|u_0\right\|_{L^2}+\left\|g\right\|_{L^2([0,T],L^2)}<\eta$, we have the following estimate
\bnan
\left\|u\right\|_{\Xsbt}\leq C(T,\eta,s)\left( \left\|u_0\right\|_{H^s}+\left\|g\right\|_{L^2([0,T],H^s)}\right)\nonumber
\enan
\end{prpstn}
\begin{proof} Using (\ref{tameestimateDuhamel}), we obtain that $u$ satisfies 
\bna
\left\|u\right\|_{\Xsbt}\leq C\left( \left\|u_0\right\|_{H^s}+\left\|g\right\|_{L^2([0,T],H^s)}\right)+C_sT^{1-b-b'}\left\|u\right\|_{X^{s,b}_T}\left(1+\left\|u\right\|_{X^{0,b}_T}^2\right)
\ena
With $T$ such that $C_sT^{1-b-b'}<1/2$, it yields
\bna
\left\|u\right\|_{\Xsbt}\leq C\left( \left\|u_0\right\|_{H^s}+\left\|g\right\|_{L^2([0,T],H^s)}\right)+C_sT^{1-b-b'}\left\|u\right\|_{X^{s,b}_T}\left\|u\right\|_{X^{0,b}_T}^2
\ena
First we use it with $s=0$. As we have proved in Lemma \ref{continuiteXsbt} the continuity with respect to $T$ of $\left\|u\right\|_{\Xzbt}$ we are in position to apply a boot-strap argument : for $T^{1-b-b'}<\frac{1}{2C_0\left( \left\|u_0\right\|_{L^2}+\left\|g\right\|_{L^2([0,T],L^2)}\right)^2}$, we obtain  : 
\bnan
\label{ineglinearbehavior2}
\left\|u\right\|_{\Xzbt}\leq C \left(\left\|u_0\right\|_{L^2}+\left\|g\right\|_{L^2([0,T],L^2)}\right).
\enan
The mass estimate (\ref{inegGronwallL2}) gives $\left\|u(t)\right\|_{L^2}\leq C\eta e^{C|t|}$. Then, we have a constant $\varepsilon(\eta,T)$ such that (\ref{ineglinearbehavior2}) holds for every interval of length smaller than $\varepsilon(\eta,T)$.
Repeating the argument on every small interval, using that $\Xzbt$ controls $L^{\infty}(L^2)$ and matching solutions with Lemma \ref{lemmerecouvrement}, we get the same result for some large interval $[0,T]$, with a constant $C$ dependent on $\eta$ and $T$. It expresses a local linear behavior.\\
Then, returning to the case $s> 0$ and $C_sT^{1-b-b'}<1/2$, we have the estimate  
$$C_sT^{1-b-b'}\left\|u\right\|_{X^{0,b}_T}^2\leq C_sT^{1-b-b'}C(\eta,T)^2\eta^2.$$
Then, for $T\leq \varepsilon(s,\eta,T)$, this can be bounded by $1/2$ and we have
\bnan
\left\|u\right\|_{\Xsbt}\leq C\left( \left\|u_0\right\|_{H^s}+\left\|g\right\|_{L^2([0,T],H^s)}\right).
\enan\\
Again, piecing solutions together , we get the same result for large $T$, with $C$ depending on $s$, $\eta$ and $T$.\end{proof}

A notable consequence of this result is that NLS has a linear behavior in any $H^s$ on any bounded set of $L^2$.

Yet, in the last estimate, the constants strongly depend on $s$. We will use the more precise estimates of the Appendix to eliminate this dependence in $s$, up to some smoother terms. 
\begin{prpstn}
\label{proplinbehaviorHsindep}
For every $T>0$, $\eta>0$ , there exists $C(T,\eta)$ such that for every $s\geq 1$, we can find $C(T,\eta,s)$ such that for every $u\in \Xsbt$ solution of (\ref{dampedeqn}) with $\left\|u_0\right\|_{L^2}+\left\|g\right\|_{L^2([0,T],L^2)}<\eta$, we have
\bnan
\label{ineglinbehavHs}
\left\|u\right\|_{\Xsbt}&\leq& C(\eta,T)\left( \left\|u_0\right\|_{H^s}+\left\|g\right\|_{L^2([0,T],H^s)}\right)\nonumber\\
&&+C(s,\eta,T)\left\|u\right\|_{X^{s-1,b}_T}\left\|u\right\|_{X^{1,b}_T}\left\|u\right\|_{X^{0,b}_T} +C(s,\eta,T) \left\|u\right\|_{X^{s-1,b}_T}.
\enan
\end{prpstn}
\begin{proof} We first assume $T\leq 1$. Lemma \ref{gainint} gives a constant $C$ independant on $s$ such that 
\bna
\left\|u\right\|_{\Xsbt}&\leq& C \left(\left\|u_0\right\|_{H^s}+\left\|g\right\|_{L^2([0,T],H^s)}\right)\\
&&+CT^{1-b-b'}\left(\left\|a^2\varphi^2u\right\|_{L^2([0,T],H^s)}+\left\|\left|u\right|^2u\right\|_{X^{s,-b'}_T}\right)
\ena
Estimate (\ref{inegtrilinavecs}) of Proposition \ref{propinegtrilin} and Corollary \ref{corfHs} of the Appendix gives some constant $C$ and $C_s$ such that  
\bna
\left\|u\right\|_{\Xsbt}&\leq& C\left( \left\|u_0\right\|_{H^s}+\left\|g\right\|_{L^2([0,T],H^s)}\right)\\
& &+ T^{1-b-b'}\left(C\left\|u\right\|_{X^{s,b}_T}+C_s\left\|u\right\|_{X^{s-1,b}_T}\right)\\
& &+ T^{1-b-b'}\left(C\left\|u\right\|^2_{X^{0,b}_T} \left\|u\right\|_{X^{s,b}_T}+C_s \left\|u\right\|_{X^{s-1,b}_T}\left\|u\right\|_{X^{1,b}_T}\left\|u\right\|_{X^{0,b}_T}\right).
\ena
From the previous Proposition, we have
$$\left\|u\right\|_{\Xzbt}\leq C(\eta,T) \left(\left\|u_0\right\|_{L^2}+\left\|g\right\|_{L^2([0,T],L^2)}\right)\leq C(\eta,T)\eta.$$
Actually, $C(\eta,T)$ can be bounded by $C(\eta)=C(\eta,1)$ if $T\leq 1$.\\
Again, for $T$ small enough (depending only on $\eta$ and not on $s$), we have
\bna
\left\|u\right\|_{\Xsbt}&\leq& C\left( \left\|u_0\right\|_{H^s}+\left\|g\right\|_{L^2([0,T],H^s)}\right)\\
&&+C_s \left\|u\right\|_{X^{s-1,b}_T}\left\|u\right\|_{X^{1,b}_T}\left\|u\right\|_{X^{0,b}_T}+C_s \left\|u\right\|_{X^{s-1,b}_T}.
\ena
Then, piecing solutions together, we finally obtain the result on a large interval $[0,T]$.\end{proof}

\begin{rmrk}
If $g=0$, the solution $u \in X^{0,b}_T$ of (\ref{dampedeqn}) actually satisfies
\bna
\left\|u(t)\right\|_{L^2}^2-\left\|u(0)\right\|_{L^2}^2=-2\intt \left\|a\varphi(\tau)u(\tau)\right\|_{L^2}^2.
\ena
\end{rmrk}
\begin{rmrk}
If $a$ is even and $u\in \Xzbt$ solution of (\ref{dampedeqn}) with  source term $g$, then $\pm u(t,-x)$ is solution with source term $\pm g(t,-x)$. As a conclusion, by uniqueness in $\Xzbt$, we infer that if $u_0$ and $g$ are odd (resp. even), then $u$ is also odd (resp. even). This gives an existence and uniqueness theorem for Dirichlet and Neumann conditions if $a\in C^{\infty}_0(]0,\pi[)$ (by identification it will become $a\in C^{\infty}(\Tu)$ even).  
\end{rmrk}
\section{Controllability near 0}
\label{sectioncontrol0}
We know (see \cite{control-nl}, \cite{control-lin1} or \cite{Machtyngier}) that any nonempty open set $\omega$ satisfies an observability estimate in $L^2$ in arbitrary small time $T>0$. Namely, for any  $a(x)\in C^{\infty}(\Tu)$ and $\varphi(t)\in C^{\infty}_0(]0,T[)$ real valued such that $a\equiv1$ on $\omega$  and $\varphi\equiv1$ on $[T/3,2T/3]$ (we add the cutoff in time to impose that the control $g$ is zero at $0$ and $T$), there exists $C>0$ such that
\bnan
\label{inegobserv}
\left\|\Psi_0 \right\|^2_{L^2} \leq C \intT \left\| a(x)\varphi(t) e^{it\partial_x^2}\Psi_0 \right\|^2_{L^2} ~dt
\enan
for every $\Psi_0\in L^2$.\\
As a consequence, using the HUM method of J-L. Lions, this implies exact controllability in $L^2$ for the linear equation. More precisely, we can follow \cite{control-nl} to construct an isomorphism of control $S$ from $L^2$ to $L^2$. For every data $\Psi_0$ in $L^2$, there exists $ \Phi_0=S^{-1}\Psi_0$, $ \Psi_0=S\Phi_0$ such that if $\Phi $ is solution of the dual equation
 \begin{eqnarray}
\label {eqlinphi}
\left\lbrace
\begin{array}{rcl}
i\partial_t \Phi + \partial_x^2 \Phi &=& 0\\
\Phi(x,0)&=&\Phi_{0}(x)
\end{array}
\right.
\end{eqnarray}
 and $\Psi$ solution of 
\begin{eqnarray}
\label {eqlin}
\left\lbrace
\begin{array}{rcl}
i\partial_t \Psi + \partial_x^2 \Psi &=&a^2(x)\varphi^2(t) \Phi\\
\Psi(T)&=&0
\end{array}
\right.
\end{eqnarray}
we have $\Psi(0)=\Psi_0$.
\begin{lmm}
$S$ is an isomorphism of $H^s$ for every $s\geq 0$.
\end{lmm}
\begin{proof} We easily see that $S$ maps $H^s$ into itself. So we just have to prove that $S\Phi_0 \in H^s$ implies $\Phi_0 \in H^s$, \emph{i.e.} $D^s\Phi_0 \in L^2$ (with notation (\ref{Dr}) of the end of the Introduction). We use the formula 
$$S \Phi_0 =i \int_0^T e^{-it\partial_x^2}\varphi^2a^2e^{it\partial_x^2} \Phi_0~dt.$$
Then, using that $S^{-1}$ is continuous from $L^2$ into itself and Lemma \ref{lemmecommut} of the Appendix, we get
\bna
\left\|D^s \Phi_0\right\|_{L^2} & \leq& C\left\|S D^s \Phi_0\right\|_{L^2}\leq  C\left\|\int_0^T e^{-it\partial_x^2}a^2\varphi^2e^{it\partial_x^2} D^s \Phi_0\right\|_{L^2}\\
&\leq & C\left\|D^s \int_0^T e^{-it\partial_x^2}a^2\varphi^2e^{it\partial_x^2}\Phi_0\right\|_{L^2}\\
&&+C\left\|\int_0^T e^{-it\partial_x^2}\left[a^2,D^s\right]\varphi^2e^{it\partial_x^2}\Phi_0\right\|_{L^2}\\
&\leq & C\left\| S \Phi_0 \right\|_{H^s}+ C_s\left\|\Phi_0\right\|_{H^{s-1}}.
\ena
This yields the desired result for $s\in [0,1]$. We obtain it for every $s\geq 0$ by iteration.\\
Moreover, if we track the dependence of each constant, especially their dependence in $s$, we get for $s\geq 1$
\bnan
\label{inegSmuunif}
\left\| S^{-1}\Psi_0 \right\|_{H^s} &\leq & C(a,\varphi,T)\left\|  \Psi_0 \right\|_{H^s}+ C(a,\varphi,s,T)\left\|\Psi_0\right\|_{H^{s-1}}.
\enan\end{proof}
\begin{thrm}
\label{thmcontrolenonlin} 
Let $\omega$ be any nonempty open subset of $\Tu$ and $T>0$. Then there exist $\varepsilon>0$ and $\eta>0$ such that for every $u_0\in L^2$ with $\left\|u_0\right\|_{L^2}<\varepsilon$, there exists $g\in C([0,T],L^2)$, with $\nor{g}{L^{\infty}([0,T],L^2)} \leq \eta$, compactly supported in $]0,T[\times \omega $ such that the unique solution $u$ in $\Xzbt$ of 
\begin{eqnarray}
\label{eqnonlinsource}
\left\lbrace
\begin{array}{rcl}
i\partial_t u + \partial_x^2 u &=&\lambda|u|^2u+g\\
u(x,0)&=&u_{0}(x)
\end{array}
\right.
\end{eqnarray}
satisfies $u(T)=0$.\\
Moreover, if $u_0\in H^s$, with $s\geq 0$, eventually with a large $H^s$ norm, we can impose $g\in C([0,T],H^s)$.
\end{thrm}
\begin{proof} We first choose $a(x) \in C^{\infty}_0(\omega)$ and $\varphi(t)\in C^{\infty}_0(]0,T[)$ different from zero, so that, observability estimate (\ref{inegobserv}) occurs. We seek $g$ under the form $\varphi^2(t)a^2(x)\Phi$ where $\Phi$ is solution of system (\ref{eqlinphi}), as in linear control theory. The purpose is then to choose the adequate $\Phi_0$  and the system is completely determined.\\
Actually, we consider the two systems
 \begin{eqnarray}
\label {eqnphi}
\left\lbrace
\begin{array}{rcl}
i\partial_t \Phi + \partial_x^2 \Phi &=& 0\\
\Phi(x,0)&=&\Phi_{0}(x)
\end{array}
\right.
\end{eqnarray}
and
\begin{eqnarray}
\label{eqnu}
\left\lbrace
\begin{array}{rcl}
i\partial_t u + \partial_x^2 u &=&\lambda|u|^2u+a^2\varphi^2\Phi\\
u(x,T)&=&0
\end{array}
\right.
\end{eqnarray}
Let us define the operator
\begin{eqnarray*}
\begin{array}{rrcl}
L:&L^2(\Tu) &\rightarrow &L^2(\Tu)\\
& \Phi_0&\mapsto &L \Phi_{0}=u_0=u(0)
\end{array}
\end{eqnarray*}
We split $u=v+\Psi$ with $\Psi$ solution of 
\begin{eqnarray}
\label{eqnPsi}
\left\lbrace
\begin{array}{rcl}
i\partial_t \Psi + \partial_x^2 \Psi &=&a^2(x)\varphi^2(t)\Phi\\
\Psi(T)&=&0
\end{array}
\right.
\end{eqnarray}
This corresponds to the linear control, and indeed $\Psi(0)=S \Phi_0$. As for function $v$, it is solution of 
\begin{eqnarray}
\label{eqnv}
\left\lbrace
\begin{array}{rcl}
i\partial_t v + \partial_x^2 v  &=&\lambda|u|^2u\\
v(T)&=&0
\end{array}
\right.
\end{eqnarray}
Then, $u$, $v$, $\Psi$ belong to $\Xzbt$ and $u(0)=v(0)+\Psi(0)$, which we can write
$$L\Phi_0= K\Phi_0+S\Phi_0$$
where $K\Phi_0=v(0)$. \\
$L \Phi_0 =u_0$ is equivalent to $\Phi_0=-S^{-1}K\Phi_0+S^{-1}u_0$. Defining the operator $B:L^2 \rightarrow L^2$ by 
$$B\Phi_0=-S^{-1}K\Phi_0+S^{-1}u_0, $$
the problem $L\Phi_0 =u_0$ is now to find a fixed point of $B$. We will prove that if $\nor{u_0}{L^2}$ is small enough, $B$ is a contraction (for the $L^2$ norm) and reproduces the closed set $$F=B_{L^2}(0,\eta) \bigcap \left(\bigcap_{i=1}^{\left\lfloor s\right\rfloor-1} B_{H^{i}}(0,R_i)\right) \bigcap  B_{H^s}(0,R_s)$$ for $\eta$ small enough and for some large $R_i$.\\
We may assume $T<1$, and fix it (actually the norm of $S^{-1}$ as an operator acting on $L^2$ or $H^s$ depends on $T$ and even explode when $T$ tends to $0$, see \cite{Phung}, \cite{Miller} and \cite{Tenenbaum}). In the rest of the proof, as we want a bound for $\eta$ independent on $s$, we will denote $C$ any constant depending only on $a$, $\varphi$, $b$, $b'$ and $T$ that are fixed. We will write $C_s$ if a dependence on $s$ is allowed.\\
Since $S$ is an isomorphism of $H^{s}$, we have
\bnan
\label{inegBHs}
\left\|B\Phi_0 \right\|_{H^s}\leq C_s\left(\left\|K\Phi_0 \right\|_{H^s}+\left\|u_0 \right\|_{H^s}\right).
\enan
So, we are led to estimate $\left\|K\Phi_0 \right\|_{H^s}=\left\|v(0) \right\|_{H^s}$.\\
Indeed, if we apply to equation (\ref{eqnv}) the same $\Xsbt$ estimates (Lemma \ref{gainint} and estimate (\ref{multilinL2}) of Lemma \ref{lemmemultilin}) we used in the existence Theorem \ref{thmexistenceNl}, we get
\bnan
\label{inegv0Hs}
\left\|v(0) \right\|_{H^s}&\leq &C \left\|v\right\|_{\Xsbt} \nonumber\\
&\leq &C T^{1-b-b'}\left\||u|^2u\right\|_{X^{s,-b'}_T} \nonumber\\
&\leq& C \left\||u|^2u\right\|_{X^{s,-b'}_T} \nonumber\\
&\leq & C_s \left\|u\right\|^2_{X^{0,b}_T}\left\|u\right\|_{X^{s,b}_T}.
\enan
Let us first consider the $L^2$ norm and use the local linear behavior of $u$ (see Proposition \ref{linearbehavior}). We obtain that for $\left\|\varphi^2 a^2 \Phi\right\|_{L^2([0,T],L^2)} \leq C\left\|\Phi_0\right\|_{L^2}<C\eta<1$,  we have 
\bna
\left\|u\right\|_{X^{0,b}_T}\leq C \left\|\Phi_0\right\|_{L^2}.
\ena
Finally, applying (\ref{inegBHs}) and (\ref{inegv0Hs}) with $s=0$, this yields 
\bna
\left\|B\Phi_0 \right\|_{L^2}\leq C\left(\left\|\Phi_0\right\|^3_{L^2}+\left\|u_0 \right\|_{L^2}\right).
\ena
Choosing $\eta$ small enough and $\left\|u_0 \right\|_{L^2}\leq \eta/2C$, we obtain $ \left\|B\Phi_0 \right\|_{L^2}\leq \eta$ and $B$ reproduces the ball $B_{\eta}$ of $L^2$.\\
For the $H^s$ norm, we distinguish two cases : $s\leq 1 $ and $s>1$.

\bigskip

For $s\leq 1$, we return to (\ref{inegv0Hs}) with the new estimate in $\Xzbt$.
\bna
\left\|v(0) \right\|_{H^s}&\leq & C_s  \eta ^2 \left\|u\right\|_{X^{s,b}_T}
\ena
$$\left\|B\Phi_0 \right\|_{H^s}\leq C_s\left(\eta ^2 \left\|u\right\|_{X^{s,b}_T}+\left\|u_0 \right\|_{H^s}\right)$$
Then, using Proposition \ref{linearbehavior} we have a linear behavior in $H^s$ norm when we have only a bounded $L^2$ norm. More precisely, for $\left\|\varphi^2a^2 \Phi\right\|_{L^2([0,T],L^2)} \leq C\left\|\Phi_0\right\|_{L^2}<C\eta<1$ we get
\bnan
\label{estimdependant}
\left\|u\right\|_{\Xsb} \leq C_s \left\|\Phi_0\right\|_{H^s}
\enan
and 
$$\left\|B\Phi_0 \right\|_{H^s}\leq C_s\left(\eta ^2 \left\|\Phi_0\right\|_{H^s}+\left\|u_0 \right\|_{H^s}\right)$$

Then, for $C_s\eta^2<1/2$, $B$ reproduces any ball in $H^s$ of radius greater than $2C_s\left\|u_0 \right\|_{H^s}$.

As a conclusion, we have proved that if $\eta <\tilde{C}_s$, $\left\|u_0 \right\|_{L^2}\leq C(\eta)$ and $R\geq C(\left\|u_0\right\|_{H^s})$, then $B$ reproduces $F$. Moreover, we can check that all the estimates are uniform for $s\leq 1$ and so the bound on $\eta$ is uniform.

\bigskip

If $s>1$, we choose the $R_i$ by induction. $R_1$ is chosen as for the case $s\leq 1$ so that $B$ reproduces $B_{H^1}(0,R_1)$. The crucial point will be to make some asumptions of smallness on $\eta$ that will be independent on $i$ and $s$. This will be possible using some estimates uniform in $s$, up to some smoother terms (that could be very large). First, we use estimate (\ref{inegSmuunif}) about $S^{-1}$.
$$\left\|B\Phi_0 \right\|_{H^i}\leq C \left\|K\Phi_0\right\|_{H^i} +C_i \left\|K\Phi_0\right\|_{H^{i-1}}+C_i \left\|u_0\right\|_{H^i}$$
The same analysis we made for the case $s\leq 1$ yields
$$\left\|K\Phi_0\right\|_{H^{i-1}}\leq C_{i-1} \eta ^2 \left\|\Phi_0\right\|_{H^{i-1}}\leq C_{i-1} \eta ^2 R_{i-1}.$$
Then, using the more precise multilinear estimate (\ref{inegtrilinavecs}) of Proposition \ref{propinegtrilin} of the Appendix, we get
\bna
\left\|v(0) \right\|_{H^i}& \leq & C \left\||u|^2u\right\|_{X^{i,-b'}_T} \\
&\leq & C \left\|u\right\|^2_{X^{0,b}_T} \left\|u\right\|_{X^{i,b}_T}+C_i \left\|u\right\|_{X^{i-1,b}_T}\left\|u\right\|_{X^{1,b}_T}\left\|u\right\|_{X^{0,b}_T}.
\ena
For the term with maximal derivative, we use the refinement (\ref{ineglinbehavHs}) of Proposition \ref{proplinbehaviorHsindep} and Corollary \ref{corfHs} of the Appendix
\bna
\left\|u\right\|_{X^{i,b}_T}&\leq& C\left\|\varphi^2a^2\Phi\right\|_{L^2([0,T],H^i)}+C_i \left\|u\right\|_{X^{i-1,b}_T}\left\|u\right\|_{X^{1,b}_T}\left\|u\right\|_{X^{0,b}_T}+C_i \left\|u\right\|_{X^{i-1,b}_T}\\
&\leq& C\left\|\Phi_0\right\|_{H^i}+C_i\left\|\Phi_0\right\|_{H^{i-1}}+C_i \left\|u\right\|_{X^{i-1,b}_T}\left\|u\right\|_{X^{1,b}_T}\left\|u\right\|_{X^{0,b}_T}\\
& &+C_i \left\|u\right\|_{X^{i-1,b}_T}.
\ena
For the terms with lower derivative, we only need estimate (\ref{estimdependant}), which  yields
\bna
\left\|v(0) \right\|_{H^i}&\leq & C \eta^2 \left\|u\right\|_{X^{i,b}_T}+C_i R_{i-1}R_1\eta\\
&\leq & C \eta^2 \left\|\Phi_0\right\|_{H^i}+C \eta^2 \left(C_i R_{i-1}+C_i R_{i-1}R_1\eta \right)+C_i R_{i-1}R_1 \eta.
\ena
Finally, we obtain
$$\left\|B\Phi_0 \right\|_{H^i}\leq C \eta^2 \left\|\Phi_0\right\|_{H^i}+C(i,\eta,R_1,R_{i-1} ,\left\|u_0\right\|_{H^i}).$$
If we choose $C\eta^2<1/2$ independant on $s$ and $R_i=2C(i,\eta,R_1,R_{i-1} ,\left\|u_0\right\|_{H^i})$, we obtain that $B$ reproduces $B_{H^{i}}(0,R_i)$. The same arguments work for $B_{H^{s}}(0,R_s)$ if $s\geq 1$.

Let us prove that $B$ is contracting for $L^2$ norm. For that, we examine the systems
\begin{eqnarray}
\label{eqnNldiff}
\left\lbrace
\begin{array}{rcl}
i\partial_t (u-\tilde{u}) + \partial_x^2 (u-\tilde{u})&=& \lambda(|u|^2u-|\tilde{u}|^2\tilde{u})+a^2\varphi^2(\Phi-\widetilde{\Phi})\\
(u-\tilde{u})(T)&=&0
\end{array}
\right.
\end{eqnarray}
\begin{eqnarray*}
\left\lbrace
\begin{array}{rcl}
i\partial_t (v-\tilde{v}) + \partial_x^2 (v-\tilde{v})  &=&\lambda(|u|^2u-|\tilde{u}|^2\tilde{u}) \\
(v-\tilde{v})(T)&=&0
\end{array}
\right.
\end{eqnarray*}
We obtain
\bnan
\label{diffB}
\left\|B\Phi_0 -B\widetilde{\Phi}_0\right\|_{L^2} & \leq & C \left\|(v-\tilde{v})(0)\right\|_{L^2} \nonumber \\
& \leq & C T^{1-b-b'}\left\||u|^2u-|\tilde{u}|^2\tilde{u}\right\|_{X^{0,-b'}_T} \nonumber\\
&\leq & C \left(\left\|u\right\|^2_{\Xzbt}+\left\|\tilde{u}\right\|^2_{\Xzbt}\right)\left\|u-\tilde{u}\right\|_{\Xzbt} \nonumber \\
&\leq& C\eta^2\left\|u-\tilde{u}\right\|_{\Xzbt}.
\enan
Considering equation (\ref{eqnNldiff}), we deduce
\bna 
\left\|u-\tilde{u}\right\|_{\Xzbt} & \leq & C T^{1-b-b'}\left\||u|^2u-|\tilde{u}|^2\tilde{u}\right\|_{X^{0,-b'}_T} + C\left\|\varphi^2a^2(\Phi-\widetilde{\Phi})\right\|_{L^2([0,T],L^2)}\\
&\leq & \left(\left\|u\right\|^2_{\Xzbt}+\left\|\tilde{u}\right\|^2_{\Xzbt}\right)\left\|u-\tilde{u}\right\|_{\Xzbt}+ C\left\| \Phi_0-\widetilde{\Phi_0}\right\|_{L^2}\\
&\leq & C \eta^2\left\|u-\tilde{u}\right\|_{\Xzbt}+ C\left\| \Phi_0-\widetilde{\Phi_0}\right\|_{L^2}.
\ena
If $\eta$ is taken small enough (independent on $s$) it yields
\bnan
\label{estimdiff}
\left\|u-\tilde{u}\right\|_{\Xzbt}\leq C \left\| \Phi_0-\widetilde{\Phi_0}\right\|_{L^2}.
\enan
Combining (\ref{estimdiff}) with (\ref{diffB}) we finally get
\bna
\left\|B\Phi_0 -B\widetilde{\Phi}\right\|_{L^2} 
&\leq& C\eta^2 \left\| \Phi_0-\widetilde{\Phi_0}\right\|_{L^2}.
\ena
Therefore, for $\eta$ small enough (independent on $s$), $B$ is a contraction of a closed set $F$ of $L^2$ and has a fixed point that by construction belongs to $H^s$. This completes the proof of Theorem \ref{thmcontrolenonlin}. \end{proof}
\begin{rmrk}
To get control for Dirichlet or Neumann condition, we have to check that if $u_0$ is odd (resp even), then the control we built is so. Suppose that $a\in C^{\infty}(\Tu)$ is even on $\Tu$ and $u_0$ is odd (resp even). Then $\check{u}(x)=-u(-x)$ is solution of (\ref{eqnu}) with $\Phi$ replaced by $\check{\Phi}(x)=-\Phi(-x)$. We have $u_0=\check{u_0}=L\check{\Phi_0}$ and indeed, $B\check{\Phi_0}=\check{\Phi_0}$. Since $\check{\Phi_0}$ has the same norm as $\Phi_0$ and by uniqueness of the fixed point in the closed set $F$, we obtain $\check{\Phi_0}=\Phi_0$ and $\Phi_0$ is odd. Therefore, the control $a^2\varphi^2\Phi$ and $u$ are odd. The same argument works similarly for $u_0$ even.
\end{rmrk}
\section{Propagation of compactness}
In this section, we adapt some theorems of Dehman-G\'erard-Lebeau \cite{control-nl} in the case of $\Xsb$ spaces. 
\begin{thrm}
\label{thmpropagation3}
Let $u_n$ be a sequence of solutions of 
$$i\partial_t u_n + \partial_x^2 u_n =f_n$$  
such that for some $0\leq b\leq 1$, we have
$$ \left\|u_n\right\|_{X^{0,b}_T} \leq C,~~\left\|u_n\right\|_{X^{-1+b,-b}_T} \rightarrow 0~~ and~~\left\|f_n\right\|_{X^{-1+b,-b}_T}\rightarrow  0 $$ 
Moreover, we assume that there is a nonempty open set $\omega$ such that $u_n \rightarrow 0$ strongly in\\ $L^2([0,T],L^2(\omega))$.\\
Then $u_n \rightarrow 0$ strongly in $L^2_{loc}([0,T],L^2(\Tu))$.
\end{thrm}
\begin{proof} 
Let $\varphi \in C^{\infty}(\Tu)$ and $\Psi \in  C^{\infty}_0(]0,T[)$ taking real values, that will be chosen later. Set $Bu=\varphi(x)D^{-1}$ and $A=\Psi(t)B$ where $D^{-1}$ is the operator defined at the end of the Introduction in (\ref{Dr}). We have $A^*=\Psi(t)D^{-1}\varphi(x)$.\\
Denote $L$ the Schr\"odinger operator $L  = i\partial_t  + \partial_x^2 $.  
For $\varepsilon >0$, we denote $A_{\varepsilon}=Ae^{\varepsilon\partial_x^2}=\Psi(t)B_{\varepsilon}$ for the regularization. We write by a classical way
\begin{eqnarray*}
\alpha_{n,\varepsilon}&=& (L u_n,A^*_{\varepsilon}u_n)_{L^2(]0,T[\times \Tu)}-(A_{\varepsilon}u_n,Lu_n,)_{L^2(]0,T[\times \Tu)}\\
&=& ([A_{\varepsilon},\partial_x^2]u_n,u_n)-i(\Psi'(t)B_{\varepsilon}u_n,u_n).
\end{eqnarray*}
But we have also
\begin{eqnarray*}
\alpha_{n,\varepsilon}&=& (f_n,A^*_{\varepsilon}u_n)_{L^2(]0,T[\times \Tu)}-(A_{\varepsilon}u_n,f_n)_{L^2(]0,T[\times \Tu)}
\end{eqnarray*}
Using Lemma \ref{lemmepseudoxsb}, we obtain
\begin{eqnarray}
\label{estimfn}
\left|(f_n,A^*_{\varepsilon}u_n)_{L^2(]0,T[\times \Tu)}\right| &\leq& \|f_n\|_{X^{-1+b,-b}_T} \|A^*_{\varepsilon}u_n\|_{X^{1-b,b}_T} \nonumber\\
&\leq& \|f_n\|_{X^{-1+b,-b}_T} \|u_n\|_{X^{0,b}_T}
\end{eqnarray}
Then, $\sup_{\varepsilon}\left|(f_n,A^*_{\varepsilon}u_n)_{L^2(]0,T[\times M)}\right|\rightarrow 0$ when $n \rightarrow \infty$. The same estimate for the other terms gives $\sup_{\varepsilon} \alpha_{n,\varepsilon} \rightarrow 0$ and likewise for the term $ (\Psi'(t)B_{\varepsilon}u_n,u_n)$.\\
Finally, taking the supremum on $\varepsilon$ tending to $0$, we get
  $$([A,\partial_x^2]u_n,u_n)_{L^2(]0,T[\times \Tu)} \rightarrow 0 \textnormal{ when } n \rightarrow \infty $$
Then, as $D^{-1}$ commutes with $\partial_x^2$, we have
\bna
[A,\partial_x^2]=-2\Psi(t)(\partial_x \varphi) \partial_x D^{-1}- \Psi(t)(\partial^2_x \varphi)D^{-1}.
\ena
Making the same estimates as in (\ref{estimfn}), we get  $$(\Psi(t)(\partial^2_x \varphi)D^{-1}u_n,u_n)_{L^2(]0,T[\times \Tu)} \rightarrow 0.$$ 
Moreover, $-i \partial_x D^{-1}$ is actually the orthogonal projection on the subspace of functions with $\widehat{u}(0)=0$. Using weak convergence, we easily obtain that $\widehat{u_n}(0)(t)$ tends to $0$ in $L^2([0,T])$ and indeed, $$(\Psi(t)(\partial_x \varphi)\widehat{u_n}(0)(t),u_n)_{L^2(]0,T[\times \Tu)} \rightarrow 0.$$
Our final result is that for any $\varphi\in C^{\infty}(\Tu)$ and $\Psi \in  C^{\infty}_0(]0,T[)$
$$(\Psi(t)(\partial_x \varphi)u_n,u_n)_{L^2(]0,T[\times \Tu)} \rightarrow 0.$$
Now, we remark that the functions that can be written $\partial_x \varphi$ are actually all the functions $\phi$ that fulfill $\int_{\Tu} \phi =0$. For example, for any $\chi \in C^{\infty}_0(\omega)$ and any $x_0\in \Tu$, $\phi(x)= \chi(x)-\chi(x-x_0)$ can be written $\phi=\partial_x \varphi$.\\
The strong convergence in $L^2([0,T],L^2(\omega))$ implies  
 $$(\Psi(t)\chi u_n,u_n)_{L^2(]0,T[\times \Tu)}\rightarrow 0.$$
Then for any $x_0\in \Tu$
$$(\Psi(t)\chi(.-x_0) u_n,u_n)_{L^2(]0,T[\times \Tu)}\rightarrow 0.$$
We close the proof by constructing a partition of unity of $\Tu$ with some functions $\chi_i(.-x^i_0) $ with $\chi_i\in C^{\infty}_0(\omega)$ and $x^i_0\in \Tu$. \end{proof}
\section{Propagation of regularity}
We write Proposition 13 of \cite{control-nl} with some $\Xsb$ assumptions on the second term of the equation.
\begin{thrm}
\label{thmpropagreg}
Let $T>0$, $0\leq b<1$ and $u\in X^{r,b}_T $, $r \in \R$ solution of
$$i \partial_t u +\partial_x^2 u=f \in X^{r,-b}_T $$
Moreover, we assume that there exists a nonempty open set $\omega$ such that $u\in L^2_{loc}(]0,T[,H^{r+\rho}(\omega))$ for some $\rho \leq \frac{1-b}{2}$.\\
Then $u\in L^2_{loc}(]0,T[,H^{r+\rho}(\Tu))$.
\end{thrm}

\begin{proof} We first regularize : $u_n=e^{\frac{1}{n}\partial_x^2}u=\Xi_n u$ and $f_n=\Xi_n u$ with $\left\|u_n\right\|_{X^{r,b}_T}\leq C $ and $\left\|f_n\right\|_{X^{r,-b}_T}\leq C $. Set $s=r+\rho$.\\
We will make a proof near the one we did for propagation of compactness.\\
Let $\varphi \in C^{\infty}(\Tu)$ and $\Psi \in  C^{\infty}_0(]0,T[)$ taking real values. Set $Bu=D^{2s-1}\varphi(x)$ and $A=\Psi(t)B$ (with notation (\ref{Dr}) of the Introduction).
If $L=i\partial_t +\partial_x^2$, we write 
\begin{eqnarray*}
& &(L u_n,A^*u_n)_{L^2(]0,T[\times \Tu)}-(Au_n,Lu_n,)_{L^2(]0,T[\times \Tu)}\\
&=& ([A,\partial_x^2]u_n,u_n)_{L^2(]0,T[\times \Tu)}-i(\Psi'(t)Bu_n,u_n)
\end{eqnarray*}
\begin{eqnarray*}
|( Au_n,f_n)_{L^2(]0,T[\times \Tu)}| &\leq & \| Au_n\|_{X^{-r,b}_T}\|f_n\|_{X^{r,-b}_T}\\
&\leq & \| u_n\|_{X^{r+2\rho-1+b,b}_T}\|f_n\|_{X^{r,-b}_T}
\end{eqnarray*}
As we have chosen
$ \rho\leq \frac{1-b}{2}$, we have $r+2\rho-1+b \leq r$. Indeed, we obtain 
\begin{eqnarray*}
|(Au_n,f_n)_{L^2(]0,T[\times \Tu)}| & \leq & C \| u_n\|_{X^{r,b}_T}\|f_n\|_{X^{r,-b}_T}\leq C
\end{eqnarray*}
The same estimates for the other terms imply that $([A,\partial_x^2]u_n,u_n)_{L^2(]0,T[\times \Tu)}$ is uniformly bounded. Yet, we have
$$[A,\partial_x^2]=-2\Psi(t)D^{2s-1} (\partial_x \varphi)\partial_x- \Psi(t)D^{2s-1}(\partial^2_x \varphi)$$
while
\begin{eqnarray*}
|(\Psi(t)D^{2s-1}(\partial^2_x \varphi)u_n,u_n)_{L^2(]0,T[\times \Tu)}| & \leq & C \| u_n\|_{X^{r,b}_T}\|u_n\|_{X^{r,-b}_T}\leq C.
\end{eqnarray*}
Finally we can control
\bnan
\label{propineg1}
|(\Psi(t)D^{2s-1} (\partial_x \varphi)\partial_x u_n,u_n)|\leq C.
\enan
If $f\in C^{\infty}_0(\omega)$ then
\bna
& &(\Psi(t)D^{2s-1} f^2 \partial_x u_n,u_n)\\
&=&(\Psi(t)D^{s-1} f \partial_x u_n, f D^s u_n)+(\Psi(t)[D^{s-1},f]\partial_x u_n, D^s u_n)\\
&=&(\Psi(t)D^{s-1} f \partial_x u_n, D^s f u_n)+(\Psi(t)D^{s-1} f \partial_x u_n,[D^s,f] u_n)\\
& &+(\Psi(t)[D^{s-1},f]\partial_x u_n, D^s u_n).
\ena
Our asumption gives $f u \in L^2_{loc}([0,T],H^s)$ and $f \partial_x u \in L^2_{loc}([0,T],H^{s-1})$. Indeed, $fu_n=\Xi_n f u+[f,\Xi_n]u$ is uniformly bounded in $L^2_{loc}([0,T],H^s)$ thanks to Lemma \ref{lemmecommut2} of Appendix and $s\leq r+1$. Making the same reasoning for $f \partial_x u_n$, we obtain
\bna
|(\Psi(t)D^{s-1} f \partial_x u_n, D^s f u_n)| \leq C.
\ena
Lemma \ref{lemmecommut} of the Appendix and $u\in L^2([0,T],H^r)$ yields (and likewise for the other term of commutator)
\bna
\left|(\Psi(t)D^{s-1} f \partial_x u_n,[D^s,f] u_n)\right|&\leq& \nor{D^{r-1} f \partial_x u_n}{L^2(L^2)}\nor{D^{\rho}[D^s,f] u_n}{L^2(L^2)}\\
&\leq& \nor{u_n}{L^2(H^{r})}\nor{u_n}{L^2(H^{s-1+\rho})}\leq C.
\ena
And finally, 
$$\left|(\Psi(t)D^{2s-1} f^2 \partial_x u_n,u_n)\right| \leq C$$
Then, writing $\partial_x \varphi= f^2(x)-f^2(x-x_0)$ and using (\ref{propineg1}), we obtain
$$\left|(\Psi(t)D^{2s-1} f^2(.-x_0) \partial_x u_n,u_n)\right| \leq C.$$
Finishing the proof as in Theorem \ref{thmpropagation3} with a partition of unity, we obtain
\bna
|(\Psi(t)D^{2s-1}\partial_x u,u)| \leq C
\ena
$$\int_0^T \sum_{k\neq 0} \Psi(t)|k|^{2s}|\widehat{u}(k,t)|^2\quad dt \leq C$$
which achieves the proof.\end{proof}

\begin{crllr}
\label{corpropagreg}
Here  $b>1/2$ and $\omega$ is any nonempty open set of $\Tu$.
Let $u\in X^{0,b}_T$ solution of 
\begin{eqnarray*}
\left\lbrace
\begin{array}{c}
i\partial_t u + \partial_x^2 u = \lambda |u|^2u\textnormal{ on } [0,T]\times \Tu\\
u \in C^{\infty}(]0,T[\times \omega) \\
\end{array}
\right.
\end{eqnarray*}
Then $u\in C^{\infty}(]0,T[\times \Tu)$
\end{crllr}

\begin{proof} 
We have $|u|^2u \in X^{0,-b}_T $ by multilinear estimates.\\
By applying once Theorem \ref{thmpropagreg}, we get $u\in L^2_{loc}([0,T],H^{1+\frac{1-b}{2}})$. Then we can choose $t_0$ such that $u(t_0) \in H^{1+\frac{1-b}{2}}$. We can then solve in $X^{1+\frac{1-b}{2},b}$ our nonlinear Schr\"odinger equation with initial data $u(t_0)$. By uniqueness in $X^{0,b}_T$, we conclude that $u\in X^{1+\frac{1-b}{2},b}_T$.\\
By iteration of this process, we get that $u\in L^2(]0,T[,H^{r})$ for every $r\in \R$ and $u\in C^{\infty}([0,T],\Tu)$.\end{proof}

\begin{crllr}
\label{uniciteL2}
Let $\omega$ be any nonempty open set of $\Tu$ and $u\in X^{0,b}_T$ solution of 
\begin{eqnarray*}
\left\lbrace
\begin{array}{rcl}
i\partial_t u + \partial_x^2 u &=& \lambda|u|^2u\textnormal{ on } [0,T]\times \Tu\\
u&=&0 \textnormal{ on } ]0,T[\times \omega \\
\end{array}
\right.
\end{eqnarray*}
Then $u=0$
\end{crllr}

\begin{proof}
Using Corollary \ref{corpropagreg}, we infer that $u\in C^{\infty}(]0,T[\times \Tu)$.\\
Proposition \ref{uniquecontinuation} of unique continuation implies $u=0$.\end{proof}
\begin{rmrk}
\label{corproluniquelin}
We have the same conclusion for $u\in X^{0,b}_T$ solution of 
\begin{eqnarray*}
\left\lbrace
\begin{array}{rcl}
i\partial_t u + \partial_x^2 u &=& 0\textnormal{ on } [0,T]\times \Tu\\
u&=&0 \textnormal{ on } ]0,T[\times \omega \\
\end{array}
\right.
\end{eqnarray*}
\end{rmrk}
\section{Stabilization}
Theorem \ref{thmstab} is a direct consequence of the following Proposition.
\begin{prpstn}
\label{propstabilisationL2}
Let $a\in L^{\infty}(\Tu)$ taking real values such that $a^2(x)>\eta$ on a nonempty open set $\omega$ of $\Tu$, for some constant $\eta>0$.\\
For every $T>0$ and every $R_0>0$, there exists a constant $C>0$ such that inequality
$$\left\|u(0)\right\|_{L^2}^2 \leq C \int_{0}^{T}\left\|au\right\|^2_{L^2} dt$$
holds for every solution $u\in \Xzbt$ of the damped equation 
\begin{eqnarray}
\label{eqndampedL2}
\left\lbrace
\begin{array}{rcl}
i\partial_t u + \partial_x^2 u +ia^2 u &=&\lambda|u|^2u \textnormal{ on } [0,T]\times \Tu\\
u(0)&=&u_{0} \in L^2
\end{array}
\right.
\end{eqnarray}
 and $\left\|u_0\right\|_{L^2}\leq R_0$.
\end{prpstn}

\begin{proof}
We argue by contradiction, we suppose the existence of a sequence $(u_n)$ of solutions of (\ref{eqndampedL2}) such that
$$\left\|u_n(0)\right\|_{L^2}\leq R_0$$
and
\begin{eqnarray}
\label{majorationdampedL2}
 \int_{0}^{T}\left\|a u_n\right\|^2_{L^2} dt \leq \frac{1}{n} 
\left\|u_{0,n}\right\|_{L^2}^2
 \end{eqnarray}
Denote $\alpha_n=\left\|u_{0,n}\right\|_{L^2} \leq R_0$. Up to extraction, we can suppose that $\alpha_n \longrightarrow \alpha$.\\
We will distinguich two cases : $\alpha>0$ and $\alpha=0$.\\
First case : $\alpha_n \longrightarrow \alpha>0$\\
By decreasing of the $L^2$ norm, $(u_n)$ is bounded in $L^{\infty}([0,T],L^2)$ and indeed in $X^{0,b}_T$. Then, as $X^{0,b}_T$ is a separable Hilbert we can extract a subsequence such that $ u_n\rightharpoonup u$ weakly in $X^{0,b}_T$ for some $u \in X^{0,b}_T$.\\
By compact embedding, as we have $b<1$ and $-b<0$, we can also extract a subsequence such that we have strong convergence in $X^{-1+b,-b}_T$.\\
At this stage, we have to be careful because as it was seen by L. Molinet in \cite{Molinet}, the weak limit $u$ is not necessarily solution to NLS. See Remark \ref{rmqmolinet} below. Thus, $\lambda|u_n|^2u_n$ is bounded in $X^{0,-b'}_T$. We can extract a subsequence such that it converges weakly in $X^{0,-b'}_T $ to some $f$ and strongly in $X^{-1+b,-b}_T$ (here, we use $b>b'$).\\
Using (\ref{majorationdampedL2}) and passing to the limit in the equation verified by $u_n$, we get
 \begin{eqnarray}
\left\lbrace
\begin{array}{rcl}
i\partial_t u + \partial_x^2 u &=&f \textnormal{ on } [0,T]\times \Tu\\
u&=&0 \textnormal{ on } [0,T]\times \omega \\
\end{array}
\right.
\end{eqnarray}
Denote $r_n=u_n-u$ and $f_n=-ia^2u_n+\lambda |u_n|^2u_n-f$, we have 
$$i\partial_t r_n + \partial_x^2 r_n =f_n$$
Moreover, because of (\ref{majorationdampedL2}) we have
$$a(x) u_n \underset{L^2([0,T],L^2)}{\longrightarrow}0$$
and so, $f_n$ converges strongly to $0$ in $X^{-1+b,-b}_T$.\\
It also implies that
$u_n \underset{L^2([0,T],L^2(\omega))}{\longrightarrow}0$ and the same for $r_n$.\\
We are then in position to apply Theorem \ref{thmpropagation3}. We infer
$$r_n \underset{L^2_{loc}([0,T],L^2)}{\longrightarrow}0.$$
Then, we can pick one $t_0 \in [0,T]$ such that $r_n(t_0)$ tends to $0$ strongly in $L^2$ and indeed $u_n(t_0)\rightarrow u(t_0) $ in $L^2$. Denote $v$ the solution of
 \begin{eqnarray}
 \label{NLSL2}
\left\lbrace
\begin{array}{rcl}
i\partial_t v + \partial_x^2 v &=&\lambda|v|^2 v \textnormal{ on } [0,T]\times \Tu\\
v(t_0)&=&u(t_0)
\end{array}
\right.
\end{eqnarray}
The main problem is, at this point, we still do not know whether $u=v$.\\
Yet, we have seen in the existence Theorem \ref{thmexistenceNl} that the flow (even backward) is Lipschitz on bounded sets. Then, as we have $u_n(t_0)\rightarrow v(t_0)$ and $ia^2 u_n \rightarrow 0$ in $L^2([0,T],L^2)$, we get $u_n \rightarrow v$ in $X^{0,b}_T$. Therefore, $u=v$ and $u$ is solution of (\ref{NLSL2}). Corollary \ref{uniciteL2} implies $u=0$.\\
In particular, we have $\left\|u_n(0)\right\|_{L^2} \rightarrow 0$ which is a contradiction to our hypothesis $\alpha >0$.

Second case : $\alpha_n \longrightarrow 0$\\
Let us make the change of unknown $v_n=u_n/\alpha_n$. $v_n$ is solution of the system
$$i\partial_t v_n +\partial_x^2 v_n +ia^2 v_n=\lambda \alpha_n^2|v_n|^2v_n$$
and
\begin{eqnarray}
\label{majorenergystabL2}
\int_{0}^{T}\left\|a v_n\right\|^2_{L^2} dt \leq \frac{1}{n}.
\end{eqnarray}
Thus, we have 
\bnan
\label{vnapprox1}
\left\|v_n(0)\right\|_{L^2}= 1
\enan
and $v_n$ is bounded in $L^{\infty}([0,T],L^2)$ as the $L^2$ norm of $u_n$ decrease.\\
By Duhamel formula and multilinear estimates, we obtain
$$\left\|v_n\right\|_{X^{0,b}_T} \leq C \left\|v_n(0)\right\|_{L^2}+CT^{1-b-b'}\left(\left\|v_n\right\|_{X^{0,b}_T}+\alpha_n^2 \left\|v_n\right\|_{X^{0,b}_T}^3\right).$$
Then, if we take $CT^{1-b-b'}<1/2$, independant of $v_n$, we have
$$\left\|v_n\right\|_{X^{0,b}_T} \leq C+C\alpha_n^2 \left\|v_n\right\|_{X^{0,b}_T}^3.$$
Lemma \ref{continuiteXsbt} states that $\left\|v_n\right\|_{X^{0,b}_T}$ is continuous in $T$. Since it is bounded near $t=0$ and $\alpha_n\rightarrow 0$, a classical boot strap argument gives that $v_n$ is bounded on $X^{0,b}_T$. Using Lemma \ref{lemmerecouvrement}, we conclude that it is bounded in $X^{0,b}_T$ even for large $T$. Therefore, $\alpha_n^2|v_n|^2v_n $ tends to $0$ in $X^{0,-b'}_T$ and indeed in $X^{-1+b,-b}_T$.\\
Then, we can extract a subsequence such that $v_n \rightharpoonup v$ in $X^{0,b}_T$ and strongly in $X^{-1+b,-b}_T$. $v$ is solution of 
\begin{eqnarray}
\left\lbrace
\begin{array}{rcl}
i\partial_t v + \partial_x^2 v &=&0 \textnormal{ on } [0,T]\times \Tu \\
v&=&0 \textnormal{ on } ]0,T[\times \omega \\
\end{array}
\right.
\end{eqnarray}
which implies $v=0$ by Remark \ref{corproluniquelin} of unique continuation.\\
Estimate (\ref{majorenergystabL2}) implies 
 $$i a^2 v_n \underset{L^2([0,T],L^2)}{\longrightarrow}0$$
and so in $X^{-1+b,-b}_T$.\\
Then, we can apply Theorem \ref{thmpropagation3} as in the first case, to get that $v_n$ converges to $0$ in $L^2_{loc}([0,T],L^2)$. Take $t_0$ such that $v_n(t_0)$ strongly converges to $0$ in $L^2$ and solve with initial data $v_n(t_0)$, we obtain that $v_n$ converges to $0$ in $X^{0,b}_T$. This contradicts (\ref{vnapprox1}).\end{proof}

\begin{rmrk}
\label{rmqmolinet}
We could have used a variant of Theorem 1 of \cite{Molinet} to get directly that the weak limit can only be zero.
\end{rmrk}
\appendix
\section{Appendix}
In this Appendix, we recall some basic microlocal analysis estimates that can be easily proved in dimension $1$, without using any general theory. We also give the proof of some multilinear Bourgain estimates.

\bigskip
Following notation (\ref{Dr}) of the Introduction, we have
\begin{lmm}
\label{lemmecommut}
Let $f$ denote the operator of multiplication by $f\in C^{\infty}(\Tu)$.\\
Then, $[D^r,f]$ maps any $H^s$ into $H^{s-r+1}$.
\end{lmm}
\begin{proof}
We denote $|~|_{\wr}$ the modification (\ref{notationnorm}) of $|~|$ introduced in Lemma \ref{lemmepseudoxsb}. We also write $\sgn(0)=1$.

We have 
\bna
\widehat{D^r(fu)}(n)=\sgn(n)\left| n\right|_{\wr}^r\sum_k \widehat{f}(n-k)\widehat{u}(k)\\
\widehat{fD^ru)}(n)=\sum_k \widehat{f}(n-k)\sgn(k)\left| k\right|_{\wr}^r\widehat{u}(k).
\ena
And then
\bna
\widehat{[D_r,f]u}(n)&=&\sum_k \widehat{f}(n-k)(\sgn(n)\left| n\right|_{\wr}^r-\sgn(k)\left| k\right|_{\wr}^r)\widehat{u}(k)\\
\left|\widehat{[D_r,f]u}(n)\right|&\leq & C\sum_k |\widehat{f}(n-k)||n-k|(\left| n\right|_{\wr}^{r-1}+\left| k\right|_{\wr}^{r-1})|\widehat{u}(k)|.
\ena
Using $\left| n\right|_{\wr}^{2\rho} \leq C\left| n-k\right|_{\wr}^{2|\rho|}\left| k\right|_{\wr}^{2\rho}$ for any $\rho \in \R$, we get
\bnan
\left\|[D_r,f]u\right\|^2_{H^{s-r+1}} & \leq & \sum_{n}\left| n\right|_{\wr}^{2s}\left(\sum_{k} \left|\widehat{f}(n-k)(n-k)\right||\widehat{u}(k)|\right)^2\nonumber\\
&+&\sum_{n}\left(\sum_{k}\left| n-k\right|_{\wr}^{|s-r+1|}\left| k\right|_{\wr}^{s} \left|\widehat{f}(n-k)(n-k)\right||\widehat{u}(k)|\right)^2\nonumber\\
&\leq &\sum_{n}\left(\sum_{k}\left| n-k\right|_{\wr}^{|s|}\left| k\right|_{\wr}^{s} \left|\widehat{f}(n-k)(n-k)\right||\widehat{u}(k)|\right)^2 \label{terme1}\\
&+&\sum_{n}\left(\sum_{k}\left|n-k\right|_{\wr}^{|s-r+1|}\left|k\right|_{\wr}^{s} \left|\widehat{f}(n-k)(n-k)\right||\widehat{u}(k)|\right)^2\label{terme2}.
\enan
We estimate (\ref{terme1}) using Cauchy-Schwarz inequality, and it is the same for (\ref{terme2}).
\bna
(\ref{terme1})&\leq &\sum_{n}\left(\sum_{k}\left| n-k\right|_{\wr}^{|s|}|\widehat{f}(n-k)(n-k)|\right) \times\\
& &\left(\sum_{k}\left|n-k\right|_{\wr}^{|s|}|\widehat{f}(n-k)(n-k)|\left| k\right|_{\wr}^{2s} |\widehat{u}(k)|^2\right)\\
&\leq &\left(\sum_{k}\left| k\right|_{\wr}^{|s|}|k\widehat{f}(k)|\right)^2\left(\sum_{k}\left| k\right|_{\wr}^{2s} |\widehat{u}(k)|^2\right)\\
&\leq & C_{f} \nor{u}{H^s}^2.
\ena \end{proof}
\begin{crllr}
\label{corfHs}
If $f\in C^{\infty}(\Tu)$, there exists some constant $C$ such that for every $s\in \R$, there exists $C_s$ such that the following estimate holds
\bna
\left\|fu\right\|_{H^s}\leq C \left\|u\right\|_{H^s}+ C_s \left\|u\right\|_{H^{s-1}}
\ena
\end{crllr}
\begin{proof} We just write $D^s(fu)=f D^s u+[D^s,f]u$.\end{proof}

\begin{lmm}
\label{lemmecommut2}
Let $f\in C^{\infty}(\Tu)$ and $\rho_{\varepsilon}=e^{\varepsilon^2 \partial^2_x}$ with $0\leq \varepsilon \leq 1$.\\
Then, $[\rho_{\varepsilon},f]$ is uniformly bounded as an operator from $H^s$ into $H^{s+1}$.
\end{lmm}
\begin{proof}
It is exactly the same as for Lemma \ref{lemmecommut} using
$$\left|e^{-\varepsilon^2 n^2}-e^{-\varepsilon^2 k^2}\right|\leq C|n-k| \left(\left\langle  n\right\rangle^{-1}+\left\langle k\right\rangle^{-1}\right)$$
because
$$\left|\partial_{\xi}\left(e^{-\varepsilon^2 \xi^2}\right)\right|\leq C \left\langle  \xi\right\rangle^{-1}.$$\end{proof}
We give the proof of multilinear Bourgain estimates. We also get some information about the dependence on $s$ of the estimates.
\begin{prpstn}
\label{propinegtrilin}
For every $s\geq 0$, we have uniformly on $T\leq 1$
\bnan
\label{inegtrilinsanss1}
 \left\| |u|^2 u \right\|_{X^{s,-3/8}_T} \leq C 3^s\left\|u\right\|^2_{X^{0,3/8}_T} \left\|u\right\|_{X^{s,3/8}_T}
\enan
\bnan
\label{inegtrilinsanss2}
\left\||u|^2u-|\tilde{u}|^2\tilde{u}\right\|_{X^{s,-3/8}_T} \leq C3^s \left(\left\|u\right\|^2_{X^{s,3/8}_T}+ \left\|\tilde{u}\right\|^2_{X^{s,3/8}_T} \right) \left\|u-\tilde{u}\right\|_{X^{s,3/8}_T}.
\enan
Moreover, there exists $C>0$ such that for every $s\geq 1$, we can find $C_s>0$ such that for every $T\leq 1$
\bnan
\label{inegtrilinavecs}
 \left\| |u|^2 u \right\|_{X^{s,-3/8}_T} &\leq& C \left\|u\right\|^2_{X^{0,3/8}_T} \left\|u\right\|_{X^{s,3/8}_T}\nonumber\\
 & &+C_s \left\|u\right\|_{X^{s-1,3/8}_T}\left\|u\right\|_{X^{1,3/8}_T}\left\|u\right\|_{X^{0,3/8}_T}.
\enan
\end{prpstn}
\begin{proof}
We follow closely \cite{Bourgainlivre} p 107.
For estimates (\ref{inegtrilinsanss1}) and (\ref{inegtrilinsanss2}), it is enough to prove
\bna
\left\|u_1 \overline{u_2}u_3\right\|_{X^{s,-3/8}}\leq C \left(\left\|u_1 \right\|_{X^{s,3/8}}\left\|u_2 \right\|_{X^{0,3/8}}\left\|u_3 \right\|_{X^{0,3/8}}\right.\\ \left.+\left\|u_1 \right\|_{X^{0,3/8}}\left\|u_2 \right\|_{X^{s,3/8}}\left\|u_3 \right\|_{X^{0,3/8}}+\left\|u_1 \right\|_{X^{0,3/8}}\left\|u_2 \right\|_{X^{0,3/8}}\left\|u_3 \right\|_{X^{s,3/8}}\right).
\ena
Denote $w=u_1 \overline{u_2}u_3$. We argue by duality. Let $v\in X^{-s,3/8}$.\\
We write $\widehat{v}(\lambda,k)$ instead of $\widehat{\widehat{v}}(\lambda,k)$ the Fourier transform in time and space variable. $\left|~\right|_{\wr}$ still denotes the modification (\ref{notationnorm}) of $|~|$ defined in the proof of Lemma \ref{lemmepseudoxsb}.
\bnan
\label{intdualite}
\int_{\R}\int_{\Tu} w \overline{v}=\sum_k \int_{\lambda} \widehat{w}(\lambda,k) \overline{\widehat{v}(\lambda,k)}
= \sum_k \int_{\lambda} \left| k\right|_{\wr}^{s}\widehat{w}(\lambda,k) \left| k\right|_{\wr}^{-s} \overline{\widehat{v}(\lambda,k)}\\
= \sum_{k_1,k_2,k_3}\int_{\lambda_1,\lambda_2,\lambda_3}\left| k\right|_{\wr}^{s}\widehat{u_1}(\lambda_1,k_1)\overline{\widehat{u_2}(\lambda_2,k_2)}\widehat{u_3}(\lambda_3,k_3)\left|k\right|_{\wr}^{-s}\overline{\widehat{v}(\lambda,k)} \nonumber
\enan
where $k=k_1-k_2+k_3$ and $\lambda=\lambda_1-\lambda_2+\lambda_3$.\\
Observe that $\left| k\right|_{\wr}^{s}\leq  3^s\max (\left| k_1\right|_{\wr}^{s},\left| k_2\right|_{\wr}^{s},\left| k_3\right|_{\wr}^{s})$. We assume $\left|k\right|_{\wr}^{s} \leq 3^s\left| k_1\right|_{\wr}^{s}$, and the other possibilities will produce the other terms of the right hand side of the estimate we want (we do not write them any more, each inequality is true if we add the same term with $u_2$ and $u_3$).
\bna
(\ref{intdualite}) \leq 3^s\sum_{k_1,k_2,k_3}\int_{\lambda_1,\lambda_2,\lambda_3}\left| k_1\right|_{\wr}^{s}\left|\widehat{u_1}(\lambda_1,k_1)\right|\left|\widehat{u_2}(\lambda_2,k_2)\right| \left|\widehat{u_3}(\lambda_3,k_3)\right|\left| k\right|_{\wr}^{-s}\left|\widehat{v}(\lambda,k)\right|
\ena
Denote $u_1^{\S}$ the function with Fourier transform equal to $\left|\widehat{u_1}(\lambda,k)\right|$. Then, using dispersive estimate (\ref{inegXsbL4})
\bna
(\ref{intdualite}) & \leq & 3^s\int_{\R} \int_{\Tu} (D^s u_1^{\S}) \overline{u_2^{\S}}u_3^{\S} \overline{D^{-s}v^{\S}}\\
&\leq &C 3^s \left\|D^s u_1^{\S}\right\|_{L^4} \left\|u_2^{\S}\right\|_{L^4}\left\|u_3^{\S}\right\|_{L^4}\left\|D^{-s}v^{\S}\right\|_{L^4}\\
&\leq & C 3^s\left\|D^s u_1^{\S}\right\|_{X^{0,3/8}} \left\|u_2^{\S}\right\|_{X^{0,3/8}}\left\|u_3^{\S}\right\|_{X^{0,3/8}}\left\|D^{-s}v^{\S}\right\|_{X^{0,3/8}}\\
&\leq & C 3^s\left\|u_1\right\|_{X^{s,3/8}} \left\|u_2\right\|_{X^{0,3/8}}\left\|u_3\right\|_{X^{0,3/8}}\left\|v\right\|_{X^{-s,3/8}}.
\ena 
Estimate (\ref{inegtrilinavecs}) is obtained similarly using the following inequality, if for example\\ $\left|k_1\right|=\max(\left|k_1\right|, \left|k_2\right|,\left|k_3\right|)$, 
$$ \left|k_1-k_2+k_3\right|_{\wr}^s \leq \left|k_1\right|_{\wr}^s + C_s\left| k_1\right|_{\wr}^{s-1} \left(\left| k_2\right|_{\wr}+\left|k_3\right|_{\wr}\right).$$
This is a consequence of the fundamental theorem of calculus applied to the function $(1+x+y)^s$.
\end{proof}


\bibliographystyle{plain}

\end{document}